\documentclass[reqno]{amsart}

\usepackage[table]{xcolor}

\definecolor{light-gray}{gray}{0.75}

\numberwithin{equation}{section}

\theoremstyle{plain}
\newtheorem{theorem}{Theorem}[section]

\newtheorem{corollary}[theorem]{Corollary}
\newtheorem{lemma}[theorem]{Lemma}
\newtheorem{proposition}[theorem]{Proposition}

\theoremstyle{definition}
\newtheorem{definition}[theorem]{Definition}

\newtheorem{example}[theorem]{Example}

\newtheorem{problem}[theorem]{Problem}
\newtheorem{question}[theorem]{Question}

\newcommand{\cref}[1]{Corollary~\ref{#1}}

\newcommand{\tref}[1]{Theorem~\ref{#1}}

\newcommand{\ldiv}{\backslash}
\newcommand{\rdiv}{/}

\newcommand{\vhi}{\varphi}
\newcommand{\gr}{\mathbf G}
\newcommand{\img}{\operatorname{Im}}
\newcommand{\inn}{\operatorname{Inn}}
\newcommand{\m}{^{-1}}
\newcommand{\Ker}{\operatorname{Ker}}
\newcommand{\aut}{\operatorname{Aut}}

\def\setof#1#2{\{#1:#2\}}
\def\genof#1#2{\langle #1:#2 \rangle}
\def\supp#1{\mathrm{supp}(#1)}
\def\lesd{\le_{\mathrm{sd}}} % subdirect product

\def\eqsymb{\varepsilon}

\begin{document}
%\title {Subdirect products of loops\\ and Moufang's theorem}
\title[Subdirect products and propagating equations]{Subdirect products and propagating equations\\ with an application to Moufang Theorem}

\author{Ale\v s Dr\'apal}
\address[Dr\'apal]{Department of Mathematics, Charles University,
Sokolovsk\'a 83, 186 75, Praha 8, Czech Republic}
\email{drapal@karlin.mff.cuni.cz}

\author{Petr Vojt\v echovsk\' y}
\address[Vojt\v{e}chovsk\'y]{Department of Mathematics, University of Denver,
2390 S.~York St, Denver, Colorado, 80208, USA}
\email{petr@math.du.edu}

\begin{abstract}
We introduce the concept of propagating equations and focus on the case of associativity propagating in varieties of loops. % fixed a typo after submission

An equation $\eqsymb$ propagates in an algebra $X$ if $\eqsymb(\overrightarrow y)$ holds whenever $\eqsymb(\overrightarrow x)$ holds and the elements of $\overrightarrow y$ are contained in the subalgebra of $X$ generated by $\overrightarrow x$. If $\eqsymb$ propagates in $X$ then it propagates in all subalgebras and products of $X$ but not necessarily in all homomorphic images of $X$. If $\mathcal V$ is a variety, the propagating core $\mathcal V_{[\eqsymb]} = \setof{X\in\mathcal V}{\eqsymb\text{ propagates in }X}$ is a quasivariety but not necessarily a variety.

We prove by elementary means Goursat's Lemma for loops and describe all subdirect products of $X^k$ and all finitely generated loops in $\mathbf{HSP}(X)$ for a nonabelian simple loop $X$. If $\mathcal V$ is a variety of loops in which associativity propagates, $X$ is a finite loop in which associativity propagates and every subloop of $X$ is either nonabelian simple or contained in $\mathcal V$, then associativity propagates in $\mathbf{HSP}(X)\lor\mathcal V$.

We study the propagating core $\mathcal S_{[x(yz)=(xy)z]}$ of Steiner loops with respect to associativity. While this is not a variety, we exhibit many varieties contained in $\mathcal S_{[x(yz)=(xy)z]}$, each providing a solution to Rajah's problem, i.e., a variety of loops not contained in Moufang loops in which Moufang Theorem holds.
\end{abstract}



\thanks{Ale\v{s} Dr\'apal partially supported by INTER-EXCELLENCE project LTAUSA19070 M\v{S}MT Czech Republic. Petr Vojt\v{e}chovsk\'y partially supported by the 2019 PROF grant of the University of Denver.}

\maketitle

\section{Introduction}

It is natural to ask whether an equation holds in a given algebra $X$ if it is satisfied on a given generating subset of $X$. The following definition formalizes this idea under the restriction that the equation involves all elements of the generating subset and it is satisfied for at least one ordering of the generators.

\begin{definition}
Let $X$ be an algebra and $\eqsymb$ an equation with $n$ variables in the signature of $X$. Then $\eqsymb$ \emph{propagates in $X$} if the implication
\begin{equation}\label{Eq:P}
    \eqsymb(x_1,\dots,x_n) \Longrightarrow ( \eqsymb(y_1,\dots,y_n)\text{ for all }y_1,\dots,y_n\in\langle x_1,\dots,x_n\rangle )
\end{equation}
holds for all $x_1,\dots,x_n\in X$. Here, $\langle x_1,\dots,x_n\rangle$ denotes the subalgebra of $X$ generated by $x_1,\dots,x_n$. An equation $\eqsymb$ \emph{propagates in a class} $\mathcal X$ of algebras if it propagates in every $X\in\mathcal X$.
\end{definition}

Some results in abstract algebra can be restated in the language of propagating equations. For instance, commutativity propagates in the variety of groups, i.e., if $xy=yx$ for some elements $x,y$ of a group, then the subgroup generated by $x$ and $y$ is commutative. The celebrated Moufang Theorem \cite{Moufang} (see \cite{Drapal} for a short proof) says that associativity propagates in the variety $\mathcal M$ of Moufang loops, i.e., if $x(yz)=(xy)z$ for some elements $x,y,z$ of a Moufang loop, then the subloop generated by $x,y,z$ is associative.

In the recent paper \cite{DV} we answered in affirmative a question of Rajah by exhibiting a variety of loops not contained in $\mathcal M$ in which associativity propagates, namely the variety of Steiner loops satisfying the identity $(xz)(((xy)z)(yz)) = ((xz)((xy)z))(yz)$. Our solution to Rajah's problem was intentionally elementary, however, our understanding of propagation of associativity in loops remains limited, not to mention propagation of general equations in universal algebras. Typical questions that come to mind are:
\begin{itemize}
\item If an equation $\eqsymb$ propagates in an algebra $X$, under which assumptions does $\eqsymb$ propagate in the variety $\mathbf{HSP}(X)$ generated by $X$?
\item If $\eqsymb$ propagates in the varieties $\mathcal V_1$ and $\mathcal V_2$, under which assumptions does $\eqsymb$ propagate in the join $\mathcal V_1\lor\mathcal V_2$?
\item Given a variety $\mathcal V$ and an equation $\eqsymb$, what is the largest variety $\mathcal W$ contained in $\mathcal V$ in which $\eqsymb$ propagates, if it exists?
\end{itemize}

In this paper we make a few initial observations about propagation of equations and then focus heavily on the special case of associativity in loops, particularly on propagation in $\mathbf{HSP}(X)$ for a given loop $X$. After introducing notation and terminology, we show that the class of algebras in which a given equation propagates is a quasivariety but not always a variety; it is therefore closed under subalgebras and products but not always under homomorphic images. In Section \ref{Sc:Basic} we derive basic properties of subdirect products of loops. In Section \ref{Sc:Lifted} we prove Goursat's Lemma for loops, showing that subdirect products of $X\times Y$ are precisely lifted graphs of isomorphisms. We also characterize normal subloops and normal subdirect products of $X\times Y$ and we describe a class of subdirect products of $X^k$. In the short Section \ref{Sc:SN} we show that the above class accounts for all subdirect products of $X^k$ as long as $X$ is a nonabelian simple loop and we describe all normal subloops of $X^k\times Y$ for an arbitrary loop $Y$. In Section \ref{Sc:Varieties} we focus on the variety $\mathbf{HSP}(X)$ generated by a single nonabelian simple loop $X$ and we describe all finitely generated algebras of $\mathbf{HSP}(X)$. Returning to propagation of equations, we prove our main result:

\emph{Suppose that an equation $\eqsymb$ propagates in a variety of loops $\mathcal V$ and in finite loops $X_1,\dots,X_n$. If every subloop $Y\le X_i$ is either in $\mathcal V$ or is nonabelian and simple, then $\eqsymb$ propagates in $\mathbf{HSP}(X_1,\dots,X_n)\lor\mathcal V$}.

Finally, in Section \ref{Sc:Steiner} we study the class $\mathcal S_{[x(yz)=(xy)z]}$ of Steiner loops  in which associativity propagates. We show that $\mathcal S_{[x(yz)=(xy)z]}$ is not a variety and we construct several varieties contained in $\mathcal S_{[x(yz)=(xy)z]}$ but not in the variety of abelian groups. Generalizing a result of Stuhl \cite{S}, we prove that an oriented anti-Pasch Steiner loop belongs to $\mathcal S_{[x(yz)=(xy)z]}$ if and only if it has exponent $4$.

We anticipate that some of our results for the variety of loops can be generalized to larger varieties. For instance, Goursat's Lemma \cite{G}, suitably interpreted, is known to hold in any Mal'cev variety \cite{CKP}.

\subsection{Notation and terminology}

See \cite{Bruck} for an introduction to loop theory, \cite{BS} for universal algebra and \cite{ColbournRosa} for triple systems.

A class of algebras with the same signature is a \emph{variety} if it is equationally defined. Given a class $\mathcal X$ of algebras, let $\mathbf H(\mathcal X)$ (resp. $\mathbf S(\mathcal X)$,
$\mathbf P(\mathcal X)$) be the class of all homomorphic images (resp. subalgebras, products) of algebras from $\mathcal X$. By Birkhoff Theorem, the smallest variety containing $\mathcal X$ is equal to $\mathbf{HSP}(\mathcal X)$.

\medskip

In a magma $(X,\cdot)$, let $L_x\colon X\to X$, $L_x(y)=xy$ and $R_x\colon X\to X$, $R_x(y)=yx$ denote the left and right translation by $x\in X$, respectively. An algebra $(X,\cdot,\ldiv,\rdiv)$ is a \emph{quasigroup} if the identities $x(x\ldiv y)=y=x\ldiv(xy)$, $(x\rdiv y)y = x = (xy)\rdiv y$ hold \cite{Evans}. A quasigroup $X$ is a \emph{loop} if there is $e\in X$ such that $ex=xe=x$ holds for all $x\in X$. Equivalently (but not from a universal-algebraic point of view, cf. \cite{BK}), $(X,\cdot,e)$ is a loop if all translations $L_x,R_x$ are permutations of $X$ and $ex=xe=e$ for all $x\in X$. Note that then $x\ldiv y = L_x^{-1}(y)$ and $x\rdiv y = R_y^{-1}(x)$.

For a nonempty subset $A$ of a loop $X$ we write $A\le X$ if $A$ is a subloop of $X$ and $A\unlhd X$ if $A$ is a normal subloop of $X$. A subloop $A$ of $X$ is \emph{proper} if $A\ne X$ and \emph{trivial} if $A=\{e\}$ or $A=X$. A loop $X$ is \emph{abelian} if it is an abelian group.

The \emph{inner mapping group} of a loop $X$ is the permutation group $\inn(X)$ generated by $L_{x,y} = L_{yx}^{-1}L_yL_x$, $R_{x,y} = R_{xy}^{-1}R_yR_x$ and $T_x=R_x^{-1}L_x$, where $x,y\in X$. The permutations $L_{x,y}$, $R_{x,y}$ and $T_x$ are known as the \emph{standard generators} of $\inn(X)$. For each of the standard generators $\psi$ of $\inn(X)$ there exists a loop term $t(x,y,z)$ such that $\psi(z)=t(x,y,z)$, where for $T_x$ we use $y$ as a dummy variable.
Each such term will be called an \emph{inner generating term}.

Note that a subloop $A$ of $X$ is normal in $X$ if and only if $t(x,y,z)\in A$ for every inner generating term $t$, every $x,y\in X$ and every $z\in A$. A loop $X$ is \emph{simple} if it has no nontrivial normal subloops.

An element $z\in X$ is \emph{central} if $t(x,y,z)=z$ for every inner generating term $t$ and every $x,y\in X$. The \emph{center} $Z(X)$ of $X$ consists of all central elements of $X$. A subloop $A\le X$ is \emph{central} if $A\le Z(X)$.

Given an abelian group $(Z,+,0)$, a loop $(F,\cdot,e)$ and a loop cocycle $f\colon F\times F\to Z$ satisfying $f(e,x)=f(x,e)=0$ for all $x\in F$, the \emph{central extension} $X=\mathrm{Ext}(Z,F,f)$ \emph{of $Z$ by $F$} is the loop $(Z\times F,*,(0,e))$ defined by $(a,x)*(b,y)=(a+b+f(x,y),xy)$. Note that $Z\times \{e\}\le Z(X)$ and $X/(Z\times \{e\})$ is isomorphic to $F$.

\medskip

A \emph{Steiner triple system of order $n$}, denoted by $STS(n)$, is a decomposition of the $n(n-1)/2$ edges of the complete graph $K_n$ into disjoint triangles, called blocks. There is a unique $STS(9)$ up to isomorphism. A \emph{Hall triple system} is a Steiner triple system in which any three elements that do not form a block generate a subsystem isomorphic to $STS(9)$. A Steiner triple system is \emph{anti-Pasch} if it contains no Pasch configurations. Note that every Hall triple system is anti-Pasch.

A \emph{Steiner quasigroup} is a commutative quasigroup satisfying the identities $xx=x$ and $x(xy)=y$. There is a one-to-one correspondence between Steiner triple systems and Steiner quasigroups defined on a set $X$: if $x\ne y$ then $x*y=z$ if and only if $\{x,y,z\}$ is a block  and $x*x=x$. A \emph{Steiner loop} is a commutative loop satisfying the identity $x(xy)=y$. There is a one-to-one correspondence between Steiner quasigroups and Steiner loops. If $(X,\cdot)$ is a Steiner quasigroup then $(X\cup\{e\},*)$ is a Steiner loop, where $x*e=e*x=x$ and $x*x=e$ for $x\in X\cup\{e\}$ and $x*y=xy$ if $x\ne y$ are elements of $X$. Conversely, if $(X\cup\{e\},*)$ is a Steiner loop then $(X,\cdot)$ is a Steiner quasigroup, where $xx=x$ and $xy=x*y$ whenever $x\ne y$.

\medskip

A mapping $\vhi\colon X\to Y$ between loops is a \emph{homomorphism} if $\vhi(xy)=\vhi(x)\vhi(y)$ for every $x,y\in X$. If $\vhi\colon X\to Y$ is a homomorphism then $\vhi(x\ldiv y) = \vhi(x)\ldiv\vhi(y)$ and
$\vhi(x\rdiv y) = \vhi(x)\rdiv\vhi(y)$ for every $x,y\in X$. We let $\img(\vhi)$ denote the image of $\vhi$ and $\Ker(\vhi)=\setof{x\in X}{\vhi(x)=e}$ the kernel of $\vhi$. Note that if $\vhi\colon X\to Y$ is a surjective loop homomorphism and $A\unlhd X$ then $f(A)\unlhd Y$ and $Y/f(A)$ is a homomorphic image of $X/A$.

\subsection{Basic properties of equation propagation}

For an equation $\eqsymb$ and a class of algebras $\mathcal X$ we call
\begin{displaymath}
    \mathcal X_{[\eqsymb]} = \setof{X\in\mathcal X}{\text{$\eqsymb$ propagates in $X$}}
\end{displaymath}
the \emph{propagating core of $X$ with respect to $\eqsymb$}.

\begin{theorem}\label{Th:Quasivariety}
Let $\eqsymb$ be an equation and $\mathcal X$ the class of all algebras in a signature $\Sigma$. Then the propagating core $\mathcal X_{[\eqsymb]}$ is a quasivariety.
\end{theorem}
\begin{proof}
Let $T$ be the class of all terms in $\Sigma$ and note that $y\in\langle x_1,\dots,x_n\rangle$ if and only if $y=u(x_1,\dots,x_n)$ for some $u\in T$. The implication \eqref{Eq:P} is therefore equivalent to the collection of implications
\begin{displaymath}
    \eqsymb(x_1,\dots,x_n) \Longrightarrow \eqsymb(u_1(x_1,\dots,x_n),\dots,u_n(x_1,\dots,x_n)),
\end{displaymath}
where $u_1,\dots,u_n$ range over $T$.
\end{proof}

\begin{corollary}\label{Cr:HS-propagation}
Let $\eqsymb$ be an equation that propagates in the class $\mathcal X$ of algebras. Then $\eqsymb$ propagates in $\mathbf S(\mathcal X)$ and in $\mathbf P(\mathcal X)$.
\end{corollary}
\begin{proof}
This follows from Theorem \ref{Th:Quasivariety} and from the fact that quasivarieties are closed under subalgebras and products, cf. \cite[Theorem V.2.25]{BS}.
\end{proof}

As an immediate consequence of Corollary \ref{Cr:HS-propagation} we obtain:

\begin{lemma}\label{Lm:PropFinGen}
An equation $\eqsymb$ propagates in an algebra $X$ if and only if it propagates in every finitely generated subalgebra of $X$.
\end{lemma}

\begin{proposition}
Let $\eqsymb$ be an equation and $\mathcal V$ a variety. Then the propagating core $\mathcal V_{[\eqsymb]}$ is a variety if and only if $\mathbf{H}(\mathcal V_{[\eqsymb]})\subseteq\mathcal V_{[\eqsymb]}$.
\end{proposition}

%\begin{proof}
%Suppose that $\mathcal V_{[s=t]}$ is a variety and $X\in\mathcal V_{[s=t]}$. Then $\mathbf H(X)\subseteq \mathcal V_{[s=t]}$ and hence $s=t$ propagates in $\mathbf{H}(X)$. Conversely, suppose that $s=t$ propagates in $\mathbf{H}(\mathcal V_{[s=t]})$ and let $X\in\mathcal V_{[s=t]}$. Then $\mathbf{HSP}(X)\subseteq \mathcal V$ since $\mathcal V$ is a variety, $s=t$ propagates in $\mathbf{SP}(X)$ by Theorem \ref{Th:Quasivariety} and in $\mathbf H(X)$ by the assumption. Hence $\mathcal V_{[s=t]}$ is closed under $\mathbf H,\mathbf S,\mathbf P$ and it is a variety.
%\end{proof}

The following example exhibits an algebra $X$ and an equation $\eqsymb$ that propagates in $X$ but not in $\mathbf{H}(X)$.

\begin{example}\label{Ex:xxx}
Consider the loop equation
\begin{equation}\label{Eq:Exp3}
    (xx)x=e
\end{equation}
and the loop $(F,\cdot,e)$ with multiplication table
\begin{displaymath}
    \begin{array}{c|ccccc}
        F&e&a&b&c&d\\
        \hline
        e&e&a&b&c&d\\
        a&a&b&d&e&c\\
        b&b&e&c&d&a\\
        c&c&d&a&b&e\\
        d&d&c&e&a&b
    \end{array}
\end{displaymath}
Then \eqref{Eq:Exp3} does not propagate in $F$ because $(aa)a = ba = e$, $b=aa\in \langle a\rangle$, but $(bb)b = cb = a\ne e$.

Let $X=\mathrm{Ext}(\mathbb Z_3,F,f) = (\mathbb Z_3\times F,*,(0,e))$, where
\begin{displaymath}
    f(x,y) = \left\{\begin{array}{ll}
        0,&\text{ if $x=e$ or $y=e$ or $x\ne y$},\\
        1,&\text{ if $x=y\ne e$}.
    \end{array}\right.
\end{displaymath}
Let $Z=\mathbb Z_3\times \{e\}\unlhd X$ and note that $X/Z$ is isomorphic to $F$. We claim that \eqref{Eq:Exp3} propagates in $X$. Consider $(r,x)\in X$. If $x\ne e$ then $((r,x)*(r,x))*(r,x) = (2r+1,x^2)*(r,x) = (1,x^2x)\ne (0,e)$, where we have used $x^2\ne x$. If $x=e$ then $(r,x)=(r,e)$ generates a subgroup of $Z\cong \mathbb Z_3$ in which \eqref{Eq:Exp3} certainly holds.
\end{example}

Example \ref{Ex:AssocDoesNotPropagate} will furnish a finite Steiner loop $X$ such that associativity propagates in $X$ but not in $\mathbf{H}(X)$.

\begin{example}\label{Ex:Ring}
Let $\mathcal R$ be the variety of unital commutative rings. We claim that $R\in\mathcal R_{[x^2=x]}$ if and only if $\mathrm{char}(R)=2$. Indeed, if idempotence propagates in $R$ then, since $1^2=1$, we must have $(1+1)^2=1+1$ and thus $\mathrm{char}(R)=2$. Conversely, suppose that $R\in\mathcal R$, $\mathrm{char}(R)=2$ and $r\in R$ satisfies $r^2=r$. Note that $\langle r\rangle = \setof{f(r)}{f\in\mathbb Z[x]} \subseteq \{0,1,r,1+r\}$. Then $u^2=u$ for every $u\in\langle r\rangle$. Note that $\mathcal R_{[x^2=x]}$ properly contains the variety $\setof{R\in\mathcal R}{x^2=x\text{ holds in }R}$ of unital boolean rings.
\end{example}

Summarizing, the propagating core $\mathcal V_{[\eqsymb]}$ of a variety $\mathcal V$ is a quasivariety but it need not be a variety, cf. Example \ref{Ex:xxx}. When $\mathcal V_{[\eqsymb]}$ is a variety, it might be properly contained between the varieties $\mathcal V$ and $\setof{X\in\mathcal V}{\eqsymb\text{ holds in }X}$, cf. Example \ref{Ex:Ring}.

\section{Basic properties of subdirect products in loops}\label{Sc:Basic}

Let $I$ be an index set, $X_i$ a loop for every $i\in I$ and $X=\prod_{i\in I}X_i$. The identity element of every $X_i$ will be denoted by $e$; it will be clear from the context to which loop the identity element $e$ belongs. Given $x=(x_i)=(x_i)_{i\in I}\in X$, we denote by $\supp{x}=\setof{i\in I}{x_i\ne e}$ the support of $x$. For $J\subseteq I$, let $p_J:X\to \prod_{i\in J}X_i$ be the canonical projection defined by $p_J((x_i)_{i\in I}) = (x_i)_{i\in J}$ and let $e_J:\prod_{i\in J}X_i\to X$ be the canonical embedding defined by $e_J((x_i)_{i\in J}) = (y_i)_{i\in I}$, where $y_i=x_i$ if $i\in J$ and $y_i=e$ otherwise. For $A\subseteq X$ and $J\subseteq I$, let
\begin{displaymath}
    A_J=p_J(A)\qquad\text{and}\qquad A[J] = \setof{x\in A}{\supp{x}\subseteq J}.
\end{displaymath}
If $A\le X$ then $A_J\le X_J$, $A[J]\le A$ and $A[J]_J = p_J(A[J])\le A_J$. When $J=\{i\}$ is a singleton, we write $p_i$, $e_i$, $A_i$ and $A[i]$ instead of $p_J$, $e_J$, $A_J$ and $A[J]$, respectively.

We say that $A\subseteq \prod_{i\in I}X_i$ is \emph{flat} if $A[i]_i=\{e\}$ for every $i\in I$.

A subloop $A\le \prod_{i\in I}X_i$ is a \emph{subdirect product} of $\prod_{i\in I}X_i$ if $A_i=X_i$ for every $i\in I$, in which case we write $A\lesd \prod_{i\in I}X_i$. Note that the factorization of $X=\prod_{i\in I}X_i$ matters in the definition of subdirect product. For instance, $A=\setof{(x,x,x)}{x\in\mathbb R}$ is a subdirect product of $\mathbb R\times\mathbb R\times\mathbb R$ but not a subdirect product of $(\mathbb R\times\mathbb R)\times\mathbb R$.

\begin{lemma}\label{partition}
Let $A\lesd \prod_{i\in I} X_i$ and let $\setof{I_k}{k\in K}$ be a partition of $I$. Then $A\lesd\prod_{k\in K}A_{I_k}$ and $A_{I_k}\lesd \prod_{i\in I_k}X_i$ for every $k\in K$.
\end{lemma}
\begin{proof}
Let $k\in K$. If $a_{I_k}\in A_{I_k}$ then there is $a\in A$ such that $p_{I_k}(a)=a_{I_k}$. This shows that $A\lesd\prod_{k\in K}A_{I_k}$. If $i\in I_k$ and $x_i\in X_i$ then there is $a\in A$ such that $p_i(a)=x_i$. But then also $p_i(p_{I_k}(a))=x_i$ and thus $A_{I_k}\lesd\prod_{i\in I_k}X_i$.
\end{proof}

\begin{lemma}\label{subloop}
Let $I$ be a finite index set and $A\le \prod_{i\in I}X_i$. Then $\prod_{i\in I}A[i]_i\le A$.
\end{lemma}
\begin{proof}
We have $A[i]_i\le A_i\le X_i$ and thus $\prod_{i\in I}A[i]_i\le X$. Let $x=(x_i)\in \prod_{i\in I}A[i]_i$. Then $x_i\in A[i]_i$ and $e_i(x_i)\in A[i]\le A$. If $I=\{1,\dots,n\}$, we conclude that $x=e_1(x_1)\cdots e_n(x_n)\in A[1]\cdots A[n]\le A$.
\end{proof}

\begin{example}
The conclusion of Lemma \ref{subloop} does not necessarily hold when $I$ is infinite. For every $i<\omega$ let $X_i=\mathbb Z_2$ and let $A$ be the subgroup of $X=\prod_{i<\omega}X_i$ consisting of all sequences with finite support. Then $A[i]_i=\mathbb Z_2$ for every $i$ but $\prod_{i\in I}A[i]_i=X$ is not contained in $A$.
\end{example}

\begin{lemma}\label{support}
Let $X= \prod_{i\in I}X_i$ and $J\subseteq I$. If $A\le X$ then $A[J]\unlhd A$ and $A[J]_J\unlhd A_J$. If $A\unlhd X$ then $A[J]\unlhd X$ and $A[J]_J\unlhd X_J$.
\end{lemma}
\begin{proof}
Let $t$ be an inner generating term. For $x,y\in X$ and $i\in I$ we have $t(x_i,y_i,e)=e$, which implies $t(x,y,X[J]) \subseteq X[J]$. If $A\le X$ then $t(A,A,A[J])\subseteq A\cap X[J]=A[J]$, proving $A[J]\unlhd A$. Since epimorphisms preserve normality, $A[J]_J\unlhd A_J$ follows. If $A\unlhd X$ then $t(X,X,A[J])\subseteq A\cap X[J]=A[J]$, so $A[J]\unlhd X$ and $A[J]_J\unlhd X_J$.
\end{proof}

\begin{lemma}\label{technical}
Let $X=\prod_{i\in I}X_i$ and let $J$, $K$ be subsets of $I$ such that all of the following conditions are satisfied:
\begin{enumerate}
\item[$\bullet$] $J$, $K$ are finite,
\item[$\bullet$] $X_j$ is finite for every $j\in J$,
\item[$\bullet$] for every $i\in I\setminus K$ there is $j\in J$ such that $X_i$ is isomorphic to $X_j$.
\end{enumerate}
Then every finitely generated subloop of $X$ is isomorphic to a subloop of $X_L$ for some finite subset $L$ of $I$.
\end{lemma}
\begin{proof}
Without loss of generality, we can assume that $X_i=X_j$ whenever $X_i$ is isomorphic to $X_j$. Let $A=\genof{a_\ell=(a_{\ell,i})}{1\le \ell\le n}$ be a finitely generated subloop of $X$. Define an equivalence relation $\sim$ on $I$ by setting $i\sim j$ if and only if $X_i=X_j$ and $a_{\ell,i}=a_{\ell,j}$ for every $1\le\ell\le n$. Since $J$, $K$ are finite and every $X_j$ with $j\in J$ is finite, $\sim$ has only finitely many equivalence classes. Let $L$ be a complete set of representatives of the equivalence classes of $\sim$. Then $p_L(A)$ is isomorphic to $A$.
\end{proof}

For $K\unlhd X$, denote by $\pi_{X/K}$ the canonical projection $X\to X/K$. A straightforward application of the Correspondence Theorem for loops yields:

\begin{proposition}\label{correspondence}
Let $X=\prod_{i\in I}{X_i}$ and $K=\prod_{i\in I}K_i$, where $K_i\unlhd X_i$ for every $i\in I$. The mapping $\pi:X\to \prod_{i\in I}X_i/K_i$ defined by $\pi((x_i)) = (x_iK_i)$ induces a lattice isomorphism between all subloops $A\le X$ containing $K$ and all subloops $B\le\prod_{i\in I}X_i/K_i$. Moreover, when $A$, $B$ are such subloops, then:
\begin{enumerate}
\item[(i)] $\pi^{-1}(B) = \setof{(x_i)}{(x_iK_i)\in B}$;
\item[(ii)] $A\unlhd X$ if and only if $\pi(A)\unlhd \prod_{i\in I}X_i/K_i$;
\item[(iii)] $A\lesd X$ if and only if $\pi(A)\lesd\prod_{i\in I}X_i/K_i$;
\item[(iv)] $\pi(A)[J] = \pi(A[J])$, $K_J\le A[J]_J$ and $(\pi(A)[J])_J = A[J]_J/K_J$ for every $J\subseteq I$.
\end{enumerate}
\end{proposition}
\begin{proof}
Let $A\le X$ and $B=\pi(X)$. We can write $\pi = \prod_{i\in I}\pi_{X_i/K_i}$ as $\pi = \gamma\pi_{X/K}$, where $\gamma:X/K\to \prod X_i/K_i$ is the isomorphism given by $(x_i)K\mapsto (x_iK_i)$. Parts (i) and (ii) then follow from the Correspondence Theorem applied to $\pi_{X/K}$.

(iii) Suppose that $A\lesd X$. Fix $j\in I$ and $a_jK_j\in X_j/K_j$. There is $(x_i)\in A$ such that $x_j=a_j$ and thus $(x_iK_i)\in B$ and $a_jK_j\in B_j$. Conversely, suppose that $B\lesd\prod X_i/K_i$. Fix $j\in I$ and $a_j\in X_j$. There is $(x_iK_i)\in B$ such that $x_jK_j=a_jK_j$ and we can assume that $a_j=x_j$. Then $(x_i)\in A$ and $a_j\in A_j$.

(iv) Let $(x_iK_i)\in \pi(A)[J]$. Then there is $a=(a_i)\in A$ such that $(a_iK_i)=(x_iK_i)$ and $a_iK_i=K_i$ for $i\not\in J$. Let $k=(k_i)$ be defined by $k_i=1$ if $i\in J$, else $k_i=a_i$.  Note that $k\in K$. Then $c=a/k\in A\cap X[J]=A[J]$ since $K\le A$ and $\pi(c)=\pi(a)=x$. Conversely, if $(x_iK_i)\in \pi(A[J])$ then there is $(a_i)\in A[J]\le A$ such that $(x_iK_i)=(a_iK_i)$, so certainly $a_iK_i=K_i$ for $i\not\in J$ and $(x_iK_i)\in \pi(A)[J]$.

If $x\in K_J$ then $x=k_J$ for some $k\in K$. Let $\ell = e_J(k_J)\in K[J]\le A[J]$. Then $\ell_J = k_J=x$ and $x\in A[J]_J$ follows. Finally, $(\pi(A)[J])_J = \pi(A[J])_J = ((A[J]K)/K)_J = (A[J]K)_J/K_J = (A[J]_J K_J)/K_J = A[J]_J/K_j$.
\end{proof}

In the situation of Proposition \ref{correspondence}, if $K\le A$ then we certainly have $K_i\le A[i]_i$ for every $i\in I$, since $K_i$ embeds into $A[i]$. Conversely, if $I$ is \emph{finite}, $A\le X$ and $K_i\le A[i]_i$ for every $i$, then $K\le A$ by Lemma \ref{subloop}.

\begin{lemma}\label{factor-by-flat}
Let $I$ be a finite index set, $A\unlhd X=\prod_{i\in I}X_i$ and $N_i=A[i]_i$ for every $i\in I$. Then $N_i\unlhd X_i$ and we can define $\pi=\prod_{i\in I}\pi_{X_i/N_i}$, $B=\pi(A)$ and $Y=\pi(X)=\prod_{i\in I}X_i/N_i$. Then:
\begin{enumerate}
\item[(i)] $B\unlhd Y$.
\item[(ii)] $B$ is a flat subloop of $Y=\prod_{i\in I}X_i/N_i$.
\item[(iii)] $A\lesd X$ if and only if $B\lesd Y$.
\item[(iv)] $X/A$ is isomorphic to $Y/B$.
\end{enumerate}
\end{lemma}
\begin{proof}
Let $K_i=N_i=A[i]_i$ and $K=\prod_{i\in I}K_i$. By Lemma \ref{subloop}, $K\le A$ since $I$ is finite. By Lemma \ref{support}, $K_i\unlhd X_i$ and thus $K\unlhd X$. By Proposition \ref{correspondence}, $B\unlhd Y$, $B[i]_i = A[i]_i/K_i = K_i/K_i = K_i$ and, also, $A\lesd X$ if and only if $B\lesd Y$. Write $\pi = \gamma\pi_{X/K}$ as in the proof of Proposition \ref{correspondence}. Using the Third Isomorphism Theorem, we then have $X/A\cong (X/K)/(A/K) = \pi_{X/K}(X)/\pi_{X/K}(A) \cong \gamma\pi_{X/K}(X)/\gamma\pi_{X/K}(A) = Y/B$.
\end{proof}

\section{Lifted isomorphism graphs and subdirect products}\label{Sc:Lifted}

Let $\vhi\colon X_1 \to X_2$ be a mapping between loops. The \emph{graph of $\vhi$} is the set
\begin{displaymath}
    \gr(\vhi) = \setof{(x,\vhi(x))}{x\in X_1}.
\end{displaymath}

\begin{lemma}\label{graph-of-hom}
Let $X_1$, $X_2$ be loops and $\vhi\colon X_1 \to X_2$ a mapping. Then $\vhi$ is a homomorphism if and only if $\gr(\vhi)\le X_1\times X_2$.
\end{lemma}
\begin{proof}
Indeed, if $x,y\in X_1$, then $(x,\vhi(x))(y,\vhi(y)) = (xy,\vhi(x)\vhi(y))$ belongs to $\gr(\vhi)$ if and only if $\vhi(x)\vhi(y)=\vhi(xy)$.
\end{proof}

\begin{lemma}\label{graph-of-inj}
Let $\vhi\colon X_1 \to X_2$ be an injective loop homomorphism. Then the following conditions are equivalent:
\begin{enumerate}
\item[(i)] $\gr(\vhi)\unlhd X_1\times X_2$,
\item[(ii)] $X_1$ is abelian and $\img(\vhi)\le Z(X_2)$,
\item[(iii)] $\gr(\vhi)\le Z(X_1\times X_2)$.
\end{enumerate}
\end{lemma}
\begin{proof}
Working in the direct product $X_1\times X_2$, we have
\begin{equation}\label{Eq:t}
    t((x,u),(y,v),(z,\vhi(z))) = (t(x,y,z),t(u,v,\vhi(z)))
\end{equation}
for every $x,y,z\in X_1$, $u,v\in X_2$ and every inner generating term $t$.

Suppose that (i) holds. Then \eqref{Eq:t} is an element of $\gr(\vhi)$ and substituting $u=v=e$, we obtain $(t(x,y,z),\vhi(z)) \in\gr(\vhi)$. Since $(z,\vhi(z))\in\gr(\vhi)$ and $\vhi$ is injective, it follows
that $t(x,y,z)=z$ and $X_1$ is abelian. With $x=y=e$ in \eqref{Eq:t}, we obtain $(z,t(u,v,\vhi(z)))\in\gr(\vhi)$, which means that $t(u,v,\vhi(z))=\vhi(z)$ and $\img(\vhi)\le Z(X_2)$.

Now suppose that (ii) holds. Then $(t(x,y,z),t(u,v,\vhi(z))) = (z,\vhi(z))$ and hence \eqref{Eq:t} shows $\gr(\vhi)\le Z(X_1\times X_2)$. Clearly, (iii) implies (i).
\end{proof}

\begin{lemma}\label{graph-iso-restricted}
Let $\vhi\colon K_1\to K_2$ be an isomorphism of loops and let $K_i\unlhd X_i$ for $i=1,2$. Then $\gr(\vhi)\unlhd X_1\times X_2$ if and only if $K_i\le Z(X_i)$ for $i=1,2$.

In particular, if $\vhi\colon X_1\to X_2$ is an isomorphism of loops, then $\gr(\vhi)\unlhd X_1\times X_2$ if and only if $X_1$ is abelian.
\end{lemma}
\begin{proof}
If $K_i\le Z(X_i)$ then $\gr(\vhi)\le K_1\times K_2\le Z(X_1)\times Z(X_2) = Z(X_1\times X_2)$ an hence $\gr(\vhi)\unlhd X_1\times X_2$. Conversely, suppose that $\gr(\vhi)\unlhd X_1\times X_2$. Then $(t(x,y,z),\vhi(z)) = (t(x,y,z),t(e,e,\vhi(z))) = t((x,e),(y,e),(z,\vhi(z)))\in \gr(\vhi)$ for every $x,y\in X_1$ and $z\in K_1$. But then $t(x,y,z)=z$ follows and we have $K_1\le Z(X_1)$. Similarly, $K_2\le Z(X_2)$.

If $\vhi\colon X_1\to X_2$ is an isomorphism, we deduce that $G(\vhi)\unlhd X_1\times X_2$ if and only if $X_i\le Z(X_i)$ for $i=1,2$, which says that $X_1$, $X_2$ are abelian. Since $X_1$ is isomorphic to $X_2$, it suffices to check that $X_1$ is abelian.
\end{proof}

For $N_1\unlhd X_1$, $N_2\unlhd X_2$ and a mapping $\vhi\colon X_1/N_1\to X_2/N_2$, let
\begin{align*}
    \gr_{X_1/N_1,X_2/N_2}(\vhi) &= (\pi_{X_1/N_1}\times \pi_{X_2/N_2})^{-1}(\gr(\vhi))\\
     &= \setof{(x_1,x_2)\in X_1\times X_2}{\vhi(x_1N_1)=x_2N_2}.
\end{align*}
If $\vhi$ is a homomorphism then $\gr_{X_1/N_1,X_2/N_2}(\vhi)$ is a subloop
of $X_1\times X_2$, being a preimage of $\gr(\vhi)$ under the homomorphism
$\pi_{X_1/N_1}\times\pi_{X_2/N_2}$.

We call a subset $A$ of $X_1\times X_2$ a \emph{lifted isomorphism graph in $X_1\times X_2$} if there exist $N_1\unlhd X_1$, $N_2\unlhd X_2$ and an isomorphism $\vhi\colon X_1/N_1\to X_2/N_2$ such that $A = \gr_{X_1/N_1,X_2/N_2}(\vhi)$. A typical lifted isomorphism graph can be visualized as follows
\begin{displaymath}
\begin{array}{|c|c|c|c|}
    \hline
    \cellcolor{light-gray}N_1\times N_2&&&\\
    \hline
    &&\cellcolor{light-gray}\phantom{N_1\times N_2}&\\
    \hline
    &&&\cellcolor{light-gray}\phantom{N_1\times N_2}\\
    \hline
    &\cellcolor{light-gray}\phantom{N_1\times N_2}&&\\
    \hline
\end{array}\,,
\end{displaymath}
where the vertical axis is indexed by cosets of $N_1$ and
the horizontal axis is indexed by cosets of $N_2$.

Given a lifted isomorphism graph $A$ in $X_1\times X_2$, it is clear that $N_1$, $N_2$ and $\varphi$ are uniquely determined. In particular, $N_i=A[i]_i$ for $i=1,2$.

\begin{proposition}[Goursat's Lemma for loops]\label{Goursat}
Let $X_1$ and $X_2$ be loops. The subdirect products of $X_1\times X_2$
are precisely the lifted isomorphism graphs in $X_1\times X_2$.

Moreover, if $A=\gr_{X_1/N_1,X_2/N_2}(\vhi)$ is a lifted isomorphism graph in $X_1\times X_2$ then $A\unlhd X_1\times X_2$ if and only if $X_1/N_1$ is abelian.
\end{proposition}
\begin{proof}
Every lifted isomorphism graph in $X_1\times X_2$ is clearly a subdirect product of $X_1\times X_2$. Conversely, let $A\lesd X_1\times X_2$, let $N_i = A[i]_i$ and note that $N_i\unlhd A_i=X_i$ by Lemma \ref{support}.

Suppose first that $N_1=\{e\}$ and $N_2=\{e\}$. If $(x_1,x_2)$, $(x_1,y_2)\in A$ then $(e,x_2/y_2) = (x_1/x_1,x_2/y_2)\in A$ and hence $x_2=y_2$ since $N_1=\{e\}$. Therefore,
for each $x_1\in X_1$ there exists exactly one $x_2\in X_2$ such that $(x_1,x_2)\in A$. Similarly, for each $x_2\in X_2$ there exists
exactly one $x_1\in X_1$ such that $(x_1,x_2)\in A$. Thus $A = \gr(\vhi)$ for some bijection $\vhi\colon X_1\to X_2$. By Lemma \ref{graph-of-hom},
$\vhi$ is an isomorphism.

In the general case, let $\pi  = \pi_{X_1/N_1}\times \pi_{X_2/N_2}$. By Proposition \ref{correspondence},
$\pi(A)\lesd X_1/N_1 \times X_2/N_2$ and $\pi(A)[i]_i = A[i]_i/N_i = N_i/N_i=N_i$.
By the previous paragraph, $\pi(A)=\gr(\vhi)$ for some isomorphism $\vhi\colon X_1/N_1\to X_2/N_2$.
Thus $A = \pi\m(\pi(A)) = \pi\m(\gr(\vhi)) = \gr_{X_1/N_1,X_2/N_2}(\vhi)$.

The last assertion follows from Lemma \ref{graph-iso-restricted} applied to $\vhi\colon X_1/N_1\to X_2/N_2$.
\end{proof}

\begin{proposition}\label{subdirect-normal}
The following conditions are equivalent for $A\le X_1\times X_2$:
\begin{enumerate}
\item[(i)] $A\unlhd X_1\times X_2$,
\item[(ii)] there exist normal subloops $M_1\unlhd X_1$, $M_2\unlhd X_2$ such that $A\lesd M_1\times M_2$, $A[i]_i = N_i\unlhd X_i$ and $M_i/N_i\le Z(X_i/N_i)$ for $i=1,2$.
\end{enumerate}
If the equivalent conditions are satisfied then $(X_1\times X_2)/A\cong (X_1/N_1\times X_2/N_2)/\gr(\vhi)$ for some isomorphism $\vhi:M_1/N_1\to M_2/N_2$.
\end{proposition}
\begin{proof}
Suppose that $A\unlhd X_1\times X_2$ and set $M_i=A_i=p_i(A)\unlhd X_i$. Then $A\lesd M_1\times M_2$. We have $N_i=A[i]_i\unlhd X_i$ by Lemma \ref{support} and obviously $N_i\le A_i=M_i$. By Goursat's Lemma, $\pi(A) = \gr(\vhi)$, where $\vhi\colon M_1/N_1\to M_2/N_2$ is some isomorphism and $\pi=\pi_{M_1/N_1}\times \pi_{M_2/N_2}$. Consider $\rho=\pi_{X_1/N_1}\times \pi_{X_2/N_2}$. By the Correspondence Theorem (or see Proposition \ref{correspondence}), $A\unlhd X_1\times X_2$ implies $\rho(A) = \pi(A) = \gr(\vhi)\unlhd X_1/N_1\times X_2/N_2$. By Lemma \ref{graph-iso-restricted}, $A_i/N_i\le Z(X_i/N_i)$ for $i=1,2$.

Conversely, suppose that (ii) holds and let $\vhi$ be the uniquely determined isomorphism $M_1/N_1\to M_2/N_2$ such that $\pi(A) = \gr(\vhi)$. Since $A_i/N_i\le Z(X_i/N_i)$ for $i=1,2$, Lemma \ref{graph-iso-restricted} implies $\gr(\vhi)\unlhd X_1/N_1\times X_2/N_2$. Then, with $\rho$ as above, we have $A = \pi^{-1}(\gr(\vhi)) = \rho^{-1}(\gr(\vhi))\unlhd X_1\times X_2$ by the Correspondence Theorem.
\end{proof}

\begin{lemma}\label{graph-in-graph}
Let $\vhi\colon X_1\to X_2$ and $\psi:X_1/N_1\to X_2/N_2$ be isomorphisms of loops and let $A=\gr_{X_1/N_1,X_2/N_2}(\psi)$. Then:
\begin{enumerate}
\item[(i)] $\gr(\vhi)\le A$ if and only if $\psi\pi_{X_1/N_1} = \pi_{X_2/N_2}\vhi$.
\item[(ii)] If $\gr(\vhi)\le A$ then $A=\setof{(xn,\vhi(x))}{x\in X_1,\,n\in N_1}$.
\item[(iii)] If $\gr(\vhi)\le A$ then $\gr(\vhi)\unlhd A$ if and only if $N_1\le Z(X_1)$.
\item[(iv)] If $\gr(\vhi)\unlhd A$ then $A/\gr(\vhi)\cong N_1\cong N_2$.
\end{enumerate}
\end{lemma}
\begin{proof}
(i) The following conditions are equivalent: $\gr(\vhi)\le A$, $\vhi(x) \in \psi(xN_1)$ for every $x\in X$, $\vhi(x)N_2 = \psi(xN_1)$ for every $x\in X_1$, $\pi_{X_2/N_2}\vhi = \psi \pi_{X_1/N_1}$. For the rest of the proof assume that $\gr(\vhi)\le A$.

(ii) If $x\in X_1$ and $n\in N_1$ then $\psi(xnN_1)=\psi(xN_1)=\vhi(x)N_2$ by (i) and thus $(xn,\vhi(x))\in A$. Conversely, let $(x_1,x_2)\in A$ and let $x=\vhi^{-1}(x_2)$. Since $(x_1,x_2)=(x_1,\vhi(x))\in A$ and $(x,\vhi(x))\in\gr(\vhi)\le A$, we have $\psi(x_1N_1)=\psi(xN_1)$, $x_1N_1=xN_1$ and $x_1 = xn$ for some $n\in N_1$.

(iii) Suppose that $\gr(\vhi)\unlhd A$. In general, if $U\unlhd V$, $u_1,u_2\in U$, $v\in V$ and $t$ is an inner generating term, then $t(u_1,u_2,v)U=vU$. Hence $(t(x,y,n),e)\gr(\vhi) = t((x,\vhi(x)),(y,\vhi(y)),(n,e))\gr(\vhi) = (n,e)\gr(\vhi)$ for all $x,y\in X_1$ and $n\in N_1$. Since $N_1\unlhd X_1$, we have $m=t(x,y,n)\in N_1$. We showed $(m,e)\in (n,e)\gr(\vhi))$, so $(m,e)=(n,e)(z,\vhi(z))=(nz,\vhi(z))$ for some $z\in X_1$. But then $z=e$, $n=m=t(x,y,n)$ and $N_1\le Z(X_1)$ follows.

Conversely, suppose that $N_1\le Z(X_1)$. For any $x,y,z\in X_1$ and $n,m\in N_1$ we have $t((xn,\vhi(x)),(ym,\vhi(m)),(z,\vhi(z))) = ( t(xn,ym,z), t(\vhi(x),\vhi(y),\vhi(z))) = (t(x,y,z),\vhi(t(x,y,z))) \in \gr(\vhi)$, where we have used $n,m\in Z(X_1)$. It follows from (ii) that $\gr(\vhi)\unlhd A$.

(iv) Suppose again that $\gr(\vhi)\unlhd A$. Then $N_1\le Z(X_1)$ by (iii). Consider $f\colon A\to N_1$ defined by $f(x,y)=x/\vhi^{-1}(y)$. For $x\in X_1$, $n\in N_1$ we have $f(xn,\vhi(x)) = (xn)/x = n$ thanks to $n\in Z(X_1)$. Then for every $x,y\in X_1$, $n,m\in N_1$, we have $f(xn,\vhi(x))f(ym,\vhi(y)) = nm = f((xy)(nm),\vhi(xy)) = f((xn)(ym),\vhi(x)\vhi(y)) = f((xn,\vhi(x))(ym,\vhi(y)))$, so $f$ is a surjective homomorphism with kernel $\gr(\vhi)$, establishing $A/\gr(\vhi)\cong N_1$. Similarly, $A/\gr(\vhi)\cong N_2$.
\end{proof}

\begin{lemma}\label{simple-1}
Let $A$ be a simple loop that is a homomorphic image of a subdirect product of $X_1\times X_2$. Then $A$ is abelian or a homomorphic image of $X_1$ or a homomorphic image of $X_2$.
\end{lemma}
\begin{proof}
Let $A\cong B/C$, where $B\lesd X=X_1\times X_2$ and $C\unlhd B$. Since $p_1:B\to X_1$ is a surjective homomorphism and $C\unlhd B$, it follows that $X_1/C_1$ is a homomorphic image of $B/C\cong A$. Since $A$ is simple, we have either $X_1/C_1\cong A$ (and we are done) or $C_1=X_1$. Similarly for the second coordinate.

We can therefore assume that $C_1=X_1$ and $C_2=X_2$, i.e., $C\lesd X_1\times X_2$. By Lemmas \ref{subloop} and \ref{support}, $K_i=C[i]_i\unlhd C$, $K=K_1\times K_2\unlhd X$ and $K\le C$. Let $\gamma:X/K\to (X_1/K_1)\times (X_2/K_2)$ be as in the proof of Proposition \ref{correspondence} so that $\pi=\pi_{X_1/K_1}\times\pi_{X_2/K_2} = \gamma\pi_{X/K}$. By the Third Isomorphism Theorem, $B/C\cong (B/K)/(C/K) = \pi_{X/K}(B)/\pi_{X/K}(C)\cong \gamma(\pi_{X/K}(B))/\gamma(\pi_{X/K}(C)) = \pi(B)/\pi(C)$. By Proposition \ref{correspondence}, $\pi(B)\lesd X_1/K_1 \times X_2/K_2$, $\pi(C) \unlhd \pi(B)$ and $(\pi(C)[i])_i = C[i]_i/K_i = K_i/K_i = K_i$ for $i=1,2$.

We can therefore assume without loss of generality that $C$ is flat. By Goursat's Lemma, $C=\gr(\vhi)$ for some isomorphism $\vhi\colon X_1\to X_2$ and $B$ is a lifted isomorphism graph in $X_1\times X_2$ such that $C\le B$. By Lemma \ref{graph-in-graph}, $A\cong B/C$ is isomorphic to $B[1]_1\le Z(X_1)$ and hence it is abelian.
\end{proof}

\begin{lemma}\label{construction}
Let $X$ be a loop, $\sim$ an equivalence relation on $\{1,\dots,k\}$ with $\ell$ equivalence classes and $\vhi_i\in\aut(X)$ for every $1\le i\le k$. Then
\begin{displaymath}
    S_X^\sim(\vhi_1,\dots,\vhi_k) =
    \setof{(x_1,\dots,x_k)\in X^k}{\vhi_i(x_i)=\vhi_j(x_j)
    \text{ whenever }i\sim j}
\end{displaymath}
is a subdirect product of $X^k$ and it is isomorphic to $X^\ell$.
\end{lemma}
\begin{proof}
Let $A = S_X^\sim(\vhi_1,\dots,\vhi_k)$. Suppose that $i\sim j$.
If $(x_1,\dots,x_k)$, $(y_1,\dots,y_k)\in A$ then $\vhi_i(x_i)=\vhi_j(x_j)$
and $\vhi_i(y_i)=\vhi_j(y_j)$ imply $\vhi_i(x_iy_i) = \vhi_j(x_jy_j)$,
so $A$ is closed under multiplication. Similarly, $A$ is closed under
divisions and hence it is a subloop of $X^k$.

Let $S$ be a complete set of representatives of the equivalence
classes of $\sim$ on $X$. For every $i\in S$, choose $x_i\in X$
arbitrarily. Then $(x_1,\dots,x_k)$ belongs to $A$ if and only if
for every $j\sim i\in S$ we have $x_j = \vhi_j^{-1}\vhi_i(x_i)$.
Hence the freely chosen tuple $(x_i)_{i\in S}$ uniquely determines
an element $(x_1,\dots,x_k)$ of $A$ and $A\cong X^{|S|}$ follows.
We can certainly arrange for any $1\le i\le k$ to be in $S$. Thus
$A$ is a subdirect product of $X^k$.
\end{proof}

When $\sim$ is the equality relation, Lemma \ref{construction} implies
that $S_X^\sim(\vhi_1,\dots,\vhi_k)\cong X^k$. When $k=2$ and
$\sim$ is the full equivalence relation, then the subdirect
products $S_X^\sim(\vhi_1,\vhi_2)$ of $X\times X$ are
precisely the graphs of automorphisms of $X$.

\section{Subdirect products and normal subloops in $X^k$ for $X$
nonabelian simple}\label{Sc:SN}

\begin{lemma}\label{simple-2}
Let $X$ be a nonabelian simple loop, $k\ge 0$ and $Y$ a loop. The normal subloops of $X^k \times Y$ are precisely the loops $M_1\times \dots \times M_k \times N$, where $M_i\in \{\{e\},X\}$ for each $i \in \{1,\dots,k\}$ and $N\unlhd Y$.
\end{lemma}
\begin{proof}
Let $A\unlhd X^k\times Y$. We proceed by induction on $k$. The case $k=0$ is clear.
Suppose that $k\ge 1$, let $X_1=X$, $X_2=X^{k-1}\times Y$ so that $A\unlhd X_1\times X_2$. We have $p_1(A)\unlhd X_1$ and $p_2(A)\unlhd X_2$. By induction, $p_2(A) = M_2\times\cdots\times M_k\times N$ for some $M_i\in\{\{e\},X\}$ and $N\unlhd Y$. Since $X$ is simple, $p_1(A)\in\{\{e\},X\}$.

If $p_1(A)=\{e\}$ then $A=\{e\}\times p_2(A)$ and we are done. Suppose that $p_1(A)=X$ so that $A\lesd X\times p_2(A)$. Note that $A\unlhd X\times p_2(A)$ because $A\unlhd X\times X_2$. By Goursat's Lemma, there are $N_1\unlhd X$ and $N_2\unlhd p_2(A)$ such that $X/N_1\cong p_2(A)/N_2$ and $X/N_1$ is abelian. Since $N_1\in\{\{e\},X\}$ and $X$ is not abelian, we must have $N_1=X$. But then $A=X\times p_2(A)$, finishing the proof.
\end{proof}

\begin{proposition}\label{subdirs-in-powers}
Let $X$ be a nonabelian simple loop and let $k$ be a positive integer. Then the subdirect products of $X^k$ are precisely the subloops $S_X^\sim(\vhi_1,\dots,\vhi_k)$ of Lemma \emph{\ref{construction}}.
\end{proposition}
\begin{proof}
We proceed by induction on $k$. The case $k=1$ is trivial. By Proposition \ref{Goursat},
a subdirect product of $X\times X$ is equal to either $X\times X$ or to
$\gr(\vhi)$ for some $\vhi\in\aut(X)$. By the remark following Lemma
\ref{construction}, these are precisely the subloops $S_X^\sim(\vhi_1,\vhi_2)$.
This gives the case $k=2$ and we can assume that $k\ge 3$.

By Lemma \ref{construction}, every $S_X^\sim(\vhi_1,\dots,\vhi_k)$ is a
subdirect product of $X^k$. Conversely, suppose that $A\lesd X^k$ and write $X^k=X_1\times X_2$ with $X_1=X$ and $X_2=X^{k-1}$.
For $1\le i\le k$ let $\delta_i$ be the homomorphism $X^k\to X^k$ that
replaces the $i$th coordinate with $e$. Then we can regard
$B_i=\delta_i(A)$ both a subloop of $X^k$ and as a subloop of $X^{k-1}$
upon forgetting the $ith$ coordinate. Note that every $B_i$ is a subdirect
product of $X^{k-1}$. In particular, by the induction assumption,
$B_1 = S_X^\sim(\vhi_2,\dots,\vhi_k)$ for some equivalence $\sim$
on $\{2,\dots,k\}$ and some automorphisms $\vhi_i$ of $X$.

Suppose first that $\sim$ is not the equality relation and
let $2\le r<s\le k$ be such that $r\sim s$. By induction assumption, $B_s=
S_X^\approx(\psi_1,\dots,\psi_{s-1},\psi_{s+1},\dots,\psi_k)$ for some
equivalence $\approx$ on $\{1,\dots,k\}\setminus\{s\}$ and some automorphisms
$\psi_i$ of $X$. Define a new equivalence relation $\equiv$ on $\{1,\dots,k\}$
by adjoining $s$ to the equivalence class $[r]_\approx$. For $i\ne s$ set
$\theta_i = \psi_i$ and let $\theta_s = \psi_r\vhi_r^{-1}\vhi_s$. We claim
that $A=S_X^\equiv(\theta_1,\dots,\theta_k)$. Note that $(x_1,\dots,x_k)\in A$
if and only if $(x_1,\dots,x_{s-1},x_{s+1},\dots,x_k)\in B_s$ and
$\vhi_r(x_r)=\vhi_s(x_s)$. If $i$, $j\in\{1,\dots,k\}\setminus\{s\}$
then $(x_1,\dots,x_{s-1},x_{s+1},\dots,x_k)\in B_s$ if and only if
$\theta_i(x_i) = \psi_i(x_i) = \psi_j(x_j)=\theta_j(x_j)$. If
$i\in [s]_\equiv$ and $i\ne s$ then $\theta_i(x_i)=\theta_s(x_s)$
if and only if $\psi_i(x_i) = \psi_r\vhi_r^{-1}\vhi_s(x_s) = \psi_r(x_r)$.

Now suppose that $\sim$ is the equality relation so that $B_1=X^{k-1}$
and $A\lesd X_1\times X_2$ with $X_1=X$ and $X_2=X^{k-1}$.
If $A=X\times X^{k-1}$, we are done. Otherwise, by
Proposition \ref{Goursat}, $A=\gr_{X_1/N_1,X_2/N_2}(\vhi)$ for some proper subloops
$N_1\unlhd X_1$, $N_2\unlhd X_2=X^{k-1}$ and some isomorphism $\vhi\colon X_1/N_1\to X_2/N_2$.
By simplicity of $X$, we have $N_1=\{e\}$, $X^{k-1}/N_2\cong X$ and
$(x_1,x_2,\dots,x_k)\in A=\gr_{X/1,X^{k-1}/N_2}(\vhi)$ iff
$\vhi(x_1) = (x_2,\dots,x_k)N_2$. By the inductive assumption on (ii),
we can assume without loss of generality that $N_2=\setof{(e,x_3,\dots,x_k)}{
x_i\in X}$. Define an equivalence $\asymp$ on $\{1,\dots,k\}$ so that $\{1,2\}$
is the only nontrivial equivalence class of $\asymp$. Let $\mu_i=1$ for $i>1$
and set $\mu_1=\rho\vhi$, where $\rho((x_2,\dots,x_k)N_2)=x_2$.
Then $(x_1,\dots,x_k)\in A$ if and only if $\mu_1(x_1) = \mu_2(x_2)$,
so $A=S_X^\asymp(\mu_1,\dots,\mu_k)$.
\end{proof}

\section{Varieties}\label{Sc:Varieties}

\begin{lemma}\label{variety-join}
Suppose that $\mathcal V_1$, $\mathcal V_2$ are varieties of loops and let $A$ be a nonabelian simple loop. If $A\not\in\mathcal V_1\cup\mathcal V_2$ then $A\not\in\mathcal V_1\lor\mathcal V_2$.
\end{lemma}
\begin{proof} Suppose that $A\in \mathcal V_1\lor\mathcal V_2$. By Lemma \ref{partition},
$A$ is a homomorphic image of a subdirect product of $X_1\times X_2$, where $X_i\in \mathcal
V_i$. By Lemma \ref{simple-1}, $A\in\mathcal V_1\cup\mathcal V_2$.
\end{proof}

\begin{theorem}\label{variety-proper}
Let $A$ be a nontrivial finite simple loop and let $\mathcal V$ be the variety
generated by all proper subloops of $A$. Then $A\notin \mathcal V$.
\end{theorem}
\begin{proof}
If $A$ is abelian then $\mathcal V$ is the variety of trivial loops and $A\not\in\mathcal V$. Suppose that $A$ is not abelian and $A\in\mathcal V$. Let $\mathcal X$ be the set of all proper subloops of $A$. Since $A$ is finite, we can assume that $A\in\mathbf H(B)$ for a finitely generated $B\in\mathbf{SP}(\mathcal X)$. Any element of $\mathbf P(\mathcal X)$ is of the form $\prod_{i\in I}X_i$, where each $X_i$ belongs to $\mathcal X$. Since $\mathcal X$ is finite, Lemma \ref{technical} applies and we can assume without loss of generality that $I$ is finite. Hence $B\lesd X_1\times\cdots\times X_k$ for suitable $X_i\in\mathcal X$. Let $k$ be as small as possible. Then $k\ge 2$ since $|B|\ge |A| > |X_1|$. By Lemma \ref{partition}, $A$ is a homomorphic image of $X\times Y$, where $X$ is a subdirect product of $X_1\times \dots \times X_{k-1}$ and $Y=X_k$. By the definition of $k$, $A$ is not a homomorphic image of $X$. By Lemma \ref{simple-1}, $A$ must be a homomorphic image of $Y$, a contradiction with $|A|>|X_k|$.
\end{proof}

\begin{proposition}\label{subdirs-in-powers-Y}
Let $X$ be a nonabelian simple loop, $k$ a positive integer and $\mathcal V$ a variety
such that $X\notin \mathcal V$. Suppose that $Y\in\mathcal V$.
Then each subdirect product of $\underbrace{X\times \cdots \times X}_k \times Y$ is isomorphic to $X^\ell\times Y$ for some $1\le\ell\le k$.
\end{proposition}
\begin{proof}
Let $X_i=X$ for $1\le i\le k$, $X_{k+1}=Y$ and $A\lesd\prod_{i=1}^{k+1}X_i$. With $I_1=\{1,\dots,k\}$ and $I_2=\{k+1\}$, Lemma \ref{partition} implies that $A\lesd A_{I_1}\times A_{I_2}$ and $A_{I_1}\lesd X^k$. By Lemma \ref{construction} and Proposition \ref{subdirs-in-powers}, every subdirect product of $X^k$ is isomorphic to $X^\ell$. Thus $A$ is (isomorphic to) a subdirect product of $X^\ell\times Y$. By Goursat's Lemma, there are $M\unlhd X^\ell$ and $N\unlhd Y$ such that $X^\ell/M\cong Y/N$. By Lemma \ref{simple-2}, $M=M_1\times\cdots\times M_\ell$, where each $M_i\in\{\{e\},X\}$. Thus $Y/N\cong X^\ell/M\cong X^r$ for some $0\le r$.
If $r>0$ then $X^r\cong Y/N\in \mathbf{HSP}(Y)\subseteq \mathcal V$ and thus $X\in\mathcal V$ (using a projection), a contradiction. Hence $r=0$, $M=X^\ell$, $N=Y$ and $A$ is isomorphic to $X^\ell\times Y$.
\end{proof}

\begin{theorem}\label{variety-proper-2}
Let $X$ be a finite nonabelian simple loop and let $\mathcal V$ be
the variety generated by all proper subloops of $X$. Then $X\notin \mathcal V$ and
each finitely generated loop in $\mathbf{HSP}(X)$ is equal to $X^k\times Y$ for some $k\ge 0$ and some finite $Y\in\mathcal V$.
\end{theorem}
\begin{proof}
By Theorem \ref{variety-proper}, $X\notin \mathcal V$. Let $A$ be a finitely
generated loop in $\mathbf{HSP}(X)$. By Lemmas~\ref{support} and~\ref{technical}, $A$
is a homomorphic image of a subdirect product of $Y_1\times \dots \times Y_n$,
where each $Y_i$ is a subloop of $X$. Hence $A$ is a homomorphic
image of a subdirect product of $X\times \dots \times X\times Y = X^k \times
Y$, where $k\ge 0$ and $Y$ is a finite loop in $\mathcal V$.
By Proposition \ref{subdirs-in-powers-Y}, $A$ is a homomorphic image of $X^\ell \times Y$,
$\ell \le k$. By Lemma \ref{simple-2}, $A$ is isomorphic to $X^h\times (Y/N)$, where $h\le \ell$
and $N\unlhd Y$.
\end{proof}

\begin{theorem}\label{variety-proper-3}
Let $X$ be a finite nonabelian simple loop. Let $\mathcal V_1$ be the variety generated by all proper subloops of $X$ and let $\mathcal V_2$ be a variety of loops not containing $X$. Let $A$ be a finitely generated loop contained in $\mathbf{HSP}(X)\lor\mathcal V_2$. Then there are $k\ge 0$ and a finitely generated loop $Y\in\mathcal V_1\lor\mathcal V_2$ such that $A\cong X^k\times Y$.
\end{theorem}
\begin{proof}
By Theorem \ref{variety-proper}, $X\not\in\mathcal V_1$. Hence $X\not\in\mathcal V_1\cup\mathcal V_2$ and $X\not\in\mathcal V_1\lor\mathcal V_2$ by Lemma \ref{variety-join}.

Note that $X$ does not belong to $\mathcal V_1\lor\mathcal V_2$ by Lemma \ref{variety-join} and \tref{variety-proper-2}. By the assumption, $A$ is a homomorphic image of a subdirect product of $U\times V$, where $U$ is a finitely generated loop in $\mathbf{HSP}(X)$ and $V\in \mathcal V_2$ is also finitely generated. By Theorem \ref{variety-proper-2}, $U = X^k \times Y$, where $Y\in \mathcal V_1$ is finite. Hence $A$ is a homomorphic image of $B\lesd Z\times W$, where $Z = X^k$ and $W\in \mathcal V_1\lor\mathcal V_2$ is finitely generated. By Goursat's Lemma, $B$ is a lifted isomorphism graph of some $\vhi:Z/N\to W/M$. By Lemma \ref{simple-2}, $Z/N\cong X^r$ for some $r\ge 0$. If $r>0$ then $X^r\cong W/M\in\mathcal V_1\lor\mathcal V_2$ implies that $X\in\mathcal V_1\lor\mathcal V_2$, a contradiction. Hence $r=0$, $\vhi$ is trivial and $B=X^k\times W$. We are done by Lemma \ref{simple-2}.
\end{proof}

Let us now return to propagation of equations.

\begin{theorem}\label{51}
Let $\mathcal V$ be a variety of loops. Let $X$ be a finite loop such that each $Y\le X$ either belongs to $\mathcal V$ or is nonabelian and simple. If an equation $\eqsymb$ propagates in both $X$ and
$\mathcal V$, then it also propagates in the variety $\mathbf{HSP}(X) \lor \mathcal V$.
\end{theorem}
\begin{proof}
We proceed by a double induction, with the outer induction on $|X|$. Let $Y_1,\dots,Y_k$ be all the subloops of $X$ listed so that $Y_i\in\mathcal V$ if and only if $1\le i\le\ell$ (for some $\ell\le k$) and $Y_k=X$. If $\ell=k$ then $X\in\mathcal V$, $\mathbf{HSP}(X)\lor\mathcal V = \mathcal V$ and we are done. We can therefore assume that $\ell<k$.

We prove by an inner induction on $\ell\le j\le k$ that $\eqsymb$ propagates in $\mathcal W_j=\mathbf{HSP}(Y_1,\dots,Y_j)\lor\mathcal V$. If $j=\ell$ then again $\mathcal W_j=\mathcal V$, so we can assume that $\ell<j\le k$ and that $\eqsymb$ propagates in $\mathcal W_{j-1}$. Note that $\mathcal W_j = \mathbf{HSP}(Y_j)\lor\mathcal W_{j-1}$ and that every subloop of $Y_j$ is either in $\mathcal W_{j-1}\supseteq\mathcal V$ or nonabelian and simple. If $j<k$ then we are done by the outer induction since $|Y_j|<|X|$. Let us therefore assume that $j=k$. By Lemma \ref{Lm:PropFinGen}, it suffices to show that $\eqsymb$ propagates in every finitely generated subloop of $\mathcal W_k = \mathbf{HSP}(X)\lor\mathcal W_{k-1}$. If $X\in\mathcal W_{k-1}$, we are done. Else the assumptions of Theorem \ref{variety-proper-3} are satisfied with $\mathcal V_2=\mathcal W_{k-1}$ and hence $A\cong X^k\times Y$ for some $Y\in\mathcal W_{k-1}$ (noting that $\mathcal V_1$ of Theorem \ref{variety-proper-3} is contained in $\mathcal W_{k-1}$). Then $\eqsymb$ propagates in $A$ by Corollary \ref{Cr:HS-propagation}.
\end{proof}

\begin{corollary}\label{Cr:Main}
Let $\mathcal V$ be a variety of loops. Let $X_1,\dots,X_n$ be finite loops such that every $Y\le X_i$ either belongs to $\mathcal V$ or is nonabelian and simple. If an equation $\eqsymb$ propagates in $X_1,\dots,X_n$ and in $\mathcal V$, then it also propagates in the variety $\mathbf{HSP}(X_1,\dots,X_n)\lor\mathcal V$.
\end{corollary}
\begin{proof}
Let $\mathcal V_i=\mathbf{HSP}(X_1,\dots,X_i)\lor\mathcal V$. By Theorem \ref{51}, $\eqsymb$ propagates in $\mathcal V_1$. If $\eqsymb$ propagates in $\mathcal V_i$, then Theorem \ref{51} with $X=X_{i+1}$ and $\mathcal V_i$ implies that $\eqsymb$ propagates in $\mathcal V_{i+1}$.
\end{proof}

\section{Steiner loops in which associativity propagates}\label{Sc:Steiner}

In this section we investigate the quasivariety $\mathcal S_{[x(yz)=(xy)z]}$ of Steiner loops in which associativity propagates. We start with two simple observations; the second one follows from \cite{CEtAl,DV}.

\begin{lemma}\label{Lm:Diassoc}
Every loop in which associativity propagates is diassociative.
\end{lemma}
\begin{proof}
Let $X$ be a loop in which associativity propagates and let $x,y\in X$. Since $x(ye)=(xy)e$, the subloop $\langle x,y,e\rangle = \langle x,y\rangle$ is associative.
\end{proof}

\begin{lemma}\label{Lm:AntiPasch}
Associativity propagates in every anti-Pasch Steiner loop.
\end{lemma}

The following example was found by Michael Kinyon using a guided finite-model builder search. It shows that $\mathcal S_{[x(yz)=(xy)z]}$ is not a variety.

\begin{example}\label{Ex:AssocDoesNotPropagate}
Let $S$ be the Steiner triple system on $13$ points $\{0,\dots,9,a,b,c\}$ with blocks (in columns)
\begin{displaymath}
\setlength{\arraycolsep}{0pt}
\begin{array}{cccccccccccccccccccccccccc}
    0& 0& 0& 0& 0& 0& 1& 1& 1& 1& 1& 2& 2& 2& 2& 2& 3& 3& 3& 4& 4& 4& 5& 5& 6& 6\\
    1& 3& 5& 7& 8& 9& 3& 4& 6& 9& a& 3& 4& 5& 7& 8& 7& 9& a& 5& 6& 8& 7& 8& 7& b\\
    2& 4& 6& c& b& a& 5& 7& 8& b& c& 6& a& c& b& 9& 8& c& b& b& 9& c& 9& a& a& c\\
\end{array}\ .
\end{displaymath}
Let $F$ be the Steiner loop corresponding to $S$, with identity element $e$. Let $f\colon F\times F\to\mathbb Z_2$ be the loop cocycle with nonzero entries only in positions
\begin{displaymath}
\setlength{\arraycolsep}{0pt}
\begin{array}{ccccccccccccccc}
    0&0&1&1&1&1&3&3&4&4&5&8&9&a&a\\
    5&6&9&a&b&c&a&b&8&c&6&c&b&b&c
\end{array}\ ,
\end{displaymath}
where a column with entries $x,y$ indicates that $f(x,y)=f(y,x)=1$. Finally, let $X=\mathrm{Ext}(\mathbb Z_2,F,f)$. Then it can be checked that $X$ is a Steiner loop, $Z(X)=\mathbb Z_2\times\{e\}$, associativity propagates in $X$ (even though $X$ is not anti-Pasch), but associativity does not propagate in $X/Z(X)$.
\end{example}

Let $\mathcal A$ be the variety of abelian groups. As an immediate consequence of Corollary \ref{Cr:Main}, we have:

\begin{corollary}\label{Cr:SpecialProp}
Let $X_1,\dots,X_n$ be finite loops in which associativity propagates and every $Y\le X_i$ is either an abelian group or a nonabelian simple loop. Then associativity propagates in $\mathbf{HSP}(X_1,\dots,X_n)\lor\mathcal A$.
\end{corollary}

Corollary \ref{Cr:SpecialProp} is a rich source of varieties of loops in which associativity propagates. For instance, both the Steiner loop of order $10$ and the unique anti-Pasch Steiner loop of order $16$ are nonabelian simple loops whose every proper subloop is abelian. More generally:

Call a Steiner loop \emph{minimal} if the corresponding Steiner triple system is minimal in the sense that each of its proper subsystems consists of at most one block.

\begin{proposition}\label{Pr:STS}
Let $S$ be a minimal anti-Pasch Steiner triple system. Then the associated Steiner loop is a nonabelian simple loop whose every proper subloop is abelian.
\end{proposition}
\begin{proof}
Let $N$ be a nontrivial normal subloop of $X$. Suppose first that $N=\{e,x\}$ so that $x\in Z(X)$. Let $y,z\in X$ be any nonidentity elements such that $\{x,y,z\}$ is not a block of $S$. Then $x(yz)\ne (xy)z$ because $S$ is anti-Pasch, a contradiction with $x\in Z(X)$. Now suppose that $|N|>2$. Since $S$ is minimal, we must have $|N|=4$ and $|X/N|=4$ as well. By the remarks in the introduction of \cite{LingEtal}, $S$ is then isomorphic to the unique anti-Pasch Steiner triple system of order $15$. An explicit calculation in the \texttt{GAP} \cite{GAP} package \texttt{LOOPS} \cite{LOOPS} shows that $X$ a nonabelian simple loop whose every proper subloop is abelian.
\end{proof}

\begin{corollary}
Let $X_1,\dots,X_n$ be Steiner loop associated with minimal anti-Pasch Steiner triple systems. Then associativity propagates in $\mathbf{HSP}(X_1,\dots,X_n)\lor\mathcal A$.
\end{corollary}
\begin{proof}
Combine Lemma \ref{Lm:AntiPasch}, Corollary \ref{Cr:SpecialProp} and Proposition \ref{Pr:STS}.
\end{proof}

We conclude the paper with a generalization of a result of Stuhl from \cite{S}, where the following definitions can also be found.

Let $S=(X,\mathcal B)$ denote a Steiner triple system with point set $X$ and blocks $\mathcal B$. A Steiner triple system $(X,\mathcal B)$ is \emph{oriented} if each of its blocks $\{x,y,z\}$ is cyclically ordered, denoted by $(x,y,z)$. We can identify the orientation of an oriented Steiner triple system $(X,\mathcal B)$ with a function $d:(X\times X)\setminus\setof{(x,x)}{x\in X}\to \mathbb Z_2=\{0,1\}$, where $d(x,y)=0$ if $(x,y,z)$ is an oriented block and $d(x,y)=1$ otherwise. In more detail, if $(x,y,z)$ is an oriented block, then $d(x,y)=d(y,z)=d(z,x)=0$ and $d(y,x)=d(z,y)=d(x,z)=1$. Note that a Steiner triple system of order $n$ gives rise to $2^{n(n-1)/6}$ oriented Steiner triple systems.

An \emph{oriented Steiner quasigroup} is a central extension $\mathrm{Ext}(\mathbb Z_2,S,f)$, where $S=(X,\cdot)$ is a Steiner quasigroup with orientation function $d$ and $f:X\times X\to\mathbb Z_2$ is any cocycle such that $f(x,y)=d(x,y)$ if $x\ne y$ and $f(x,x)=f(y,y)$ for all $x,y\in X$. Thus an oriented Steiner triple system gives rise to two oriented Steiner quasigroups, depending on the value of $f(x,x)\in\mathbb Z_2$. Note that an oriented Steiner quasigroup is \emph{not} a Steiner quasigroup. Indeed, we have $(a,x)*(a,x) = (a+a+f(x,x),xx) = (f(x,x),x)$, so either $(0,x)*(0,x)\ne (0,x)$ or $(1,x)*(1,x)\ne (1,x)$.

If $\mathrm{Ext}(\mathbb Z_2,S,f)$ is an oriented Steiner quasigroup and $L$ is the Steiner loop associated with $S$, then $X=\mathrm{Ext}(\mathbb Z_2,L,f)$ will be called an \emph{oriented Steiner loop}, where we extend the domain of the cocycle $f:S\times S\to\mathbb Z_2$ to $L\times L$ by setting $f(x,e)=f(e,x)=0$ for every $x\in L$. If $f(x,x)=0$ for every $x\in S$ then $(a,x)*(a,x) = (a+a+f(x,x),xx)=(0,e)$ and $X$ has exponent $2$. If $f(x,x)=1$ for every $x\in S$ then $(a,x)^4=(0,e)$ no matter how $(a,x)$ is parenthesized and thus $X$ has exponent $4$.

Stuhl proved in \cite[Theorem 1 and Corollary 2]{S} that an oriented Hall loop satisfies Moufang Theorem if and only if it is of exponent $4$. We generalize her result in Theorem \ref{Th:GenStuhl}.

\begin{lemma}\label{Lm:Assoc}
Let $X=\mathrm{Ext}(\mathbb Z_2,L,f)$ be an oriented anti-Pasch Steiner loop and let $d=f(x,x)$ for some (and hence all) $x\in L\setminus\{e\}$. Then
\begin{displaymath}
    (a,x)*((b,y)*(c,z)) = (a,x)*((b,y)*(c,z))
\end{displaymath}
if and only if one of the following conditions is satisfied:
\begin{itemize}
\item $e\in\{x,y,z\}$,
\item $e\ne x=y$ and $d=1$,
\item $e\ne y=z$ and $d=1$,
\item $x=z$,
\item $\{x,y,z\}$ is a block.
\end{itemize}
\end{lemma}
\begin{proof}
We have $(a,x)*((b,y)*(c,z)) = (a,x)*((b,y)*(c,z))$ if and only if both
\begin{align}
    xy\cdot z &= x\cdot yz,\label{ca}\\
    f(x,y)+f(xy,z) &= f(x,yz)+f(y,z)\label{cb}
\end{align}
hold. If $e\in\{x,y,z\}$ then both conditions hold, so we can assume from now on that $e\not\in\{x,y,z\}$.

If $x=y$ or $y=z$ or $x=z$ then \eqref{ca} holds. Note that if $x=y=z$ then \eqref{cb} reduces to $f(x,x)+f(e,x) = f(x,e)+f(x,x)$, i.e., to $d+0=0+d$, so it also holds. Suppose from now on that $|\{x,y,z\}|>1$.

If $x=y$ then \eqref{cb} becomes $d = f(x,x)+f(e,z) = f(x,xz) + f(x,z)$. Whether $(x,z,xz)$ is a block or $(x,xz,z)$ is a block, the right hand side reduces to $1$, so the three elements associate if and only if $d=1$. Similarly, if $y=z$ then \eqref{cb} becomes $f(x,y)+f(xy,y) = f(x,e) + f(y,y) = d$ and the left hand side is equal to $1$ whether $(x,y,xy)$ is a block or $(x,xy,y)$ is a block. Finally, if $x=z$ then \eqref{cb} becomes $f(x,y)+f(xy,x) = f(x,yx) + f(y,x)$, which holds whether $(x,y,xy)$ is a block or $(x,xy,y)$ is a block.

We can therefore assume that $|\{x,y,z\}|=3$. If $\{x,y,z\}$ is a block then we can take $z=xy$, \eqref{ca} holds and \eqref{cb} becomes $f(x,y)+f(xy,xy) = f(x,x)+f(y,xy)$, which holds whether $(x,y,xy)$ is a block or $(x,xy,y)$ is a block. If $\{x,y,z\}$ is not a block then $xy\cdot z\ne x\cdot yz$ because $L$ is anti-Pasch.
\end{proof}

\begin{theorem}\label{Th:GenStuhl}
Let $X$ be an oriented anti-Pasch Steiner loop. Then associativity propagates in $X$ if and only if $X$ has exponent $4$.
\end{theorem}
\begin{proof}
Let $X=\mathrm{Ext}(\mathbb Z_2,L,f)$ for some anti-Pasch Steiner loop $L$ and let $d=f(x,x)$ for some $x\in L\setminus\{e\}$. Note that if $x\ne y$ are nonidentity elements of $L$, then $\mathbb Z_2\times\{e,x,y,xy\}$ is a group if and only if $d=1$, by Lemma \ref{Lm:Assoc}.

Suppose now that $((a,x)*(b,y))*(c,z) = (a,x)*((b,y)*(c,z))$ for some elements of $X$. In all five cases of Lemma \ref{Lm:Assoc}, we can choose $u,v\in L$ so that $\langle (a,x),(b,y),(c,z)\rangle \le \mathbb Z_2\times\{e,u,v,uv\}$. Hence, if $d=1$ then associativity propagates in $X$. If $d=0$, consider any nonidentity elements $x\ne y\in L$. By Lemma \ref{Lm:Assoc}, $(0,x)*((0,x)*(0,y))\ne ((0,x)*(0,x))*(0,y)$ and $X$ is not diassociative.  We are done by Lemma \ref{Lm:Diassoc}.
\end{proof}

\section{Open problems}

We start with concrete problems concerning propagation of associativity in Steiner loops. By Example \ref{Ex:AssocDoesNotPropagate}, there exists a Steiner loop $X$ of order $28$ such that associativity propagates in $X$ but not in $\mathbf{H}(X)$.

\begin{problem}
What is the smallest order of a Steiner loop $X$ such that associativity propagates in $X$ but not in $\mathbf{H}(X)$?
\end{problem}

The same example shows that a Steiner loop in which associativity does not propagate might be a factor of a Steiner loop in which associativity propagates.

\begin{problem}\label{Pr:HomImageOfProp}
Is it true that for every Steiner loop $F$ (in which associativity does not propagate) there exists a Steiner loop $X$ in which associativity propagates and $F\in\mathbf{H}(X)$?
\end{problem}

By Corollary \ref{Cr:SpecialProp}, if associativity propagates in a Steiner loop $X$ and all subloops of $X$ are either abelian or nonabelian simple, then it propagates in $\mathbf{HSP}(X)$, too.

\begin{problem}
Characterize the class of Steiner loops $X$ for which associativity propagates in $\mathbf{HSP}(X)$.
\end{problem}

The discussion of Sections \ref{Sc:Basic}--\ref{Sc:Varieties} has been intentionally restricted to loops since we believe that the explicit description of subloops we have presented may be useful in the future. We expect that some of the results in these sections generalize to Mal'cev varieties and possibly to Goursat varieties.

\medskip

We now turn to more general questions concerning propagation of equations. These can be seen as suggestions for research programs. Recall that
\begin{displaymath}
    \mathcal V_{[\eqsymb]}=\setof{X\in\mathcal V}{\eqsymb\text{ propagates in }X}.
\end{displaymath}

\begin{question}
Given an equation $\eqsymb$ and a variety $\mathcal V$, when is the propagating core $\mathcal V_{[\eqsymb]}$ a variety? If $\mathcal V_{[\eqsymb]}$ is a variety, is it finitely based relative to $\mathcal V$?
\end{question}

\begin{question}
If $\eqsymb$ propagates in $X$, under which conditions does $\eqsymb$ propagate in $\mathbf{HSP}(X)$?
\end{question}

\begin{question}
Let $\mathcal V_i$, $i\in I$, be varieties in which $\eqsymb$ propagates. Under which conditions does $\eqsymb$ propagate in the join $\bigvee_{i\in I}\mathcal V_i$?
\end{question}

The following question generalizes Problem \ref{Pr:HomImageOfProp}.

\begin{question}
Given an equation $\eqsymb$ and a variety $\mathcal V$, under which conditions is $\mathcal V\subseteq \mathbf{H}(\mathcal V_{[\eqsymb]})$?
\end{question}

\begin{question}\label{Qu:last}
Given an equation $\eqsymb$ and a variety $\mathcal V$, is there an algebra $X\in\mathcal V$ such that $\mathbf{H}(X)\subseteq\mathcal V_{[\eqsymb]}$ but $\mathbf{HSP}(X)\not\subseteq\mathcal V_{[\eqsymb]}$?
\end{question}

Answering Question \ref{Qu:last} for a given $\eqsymb$ and $\mathcal V$ might not be difficult but it might be technically complicated. A possible strategy is to first find $Y,Z$ such that $Y\in\mathcal V_{[\eqsymb]}$, $Z\in\mathbf{H}(Y)$ but $Z\not\in\mathcal V_{[\eqsymb]}$, then embed $Y$ into a simple algebra $X\in\mathcal V_{[\eqsymb]}$. Then $\mathbf{H}(X)\subseteq\mathcal V_{[\eqsymb]}$ by simplicity but $Z\in\mathbf{HS}(X)\not\subseteq \mathcal V_{[\eqsymb]}$.

\end{document}